\documentclass[11pt]{amsart}
\usepackage{amsmath}
\usepackage{amsfonts}
\usepackage{amssymb,bm}
\usepackage{amsthm}
\usepackage{amscd}
\usepackage{amsbsy}
\usepackage{a4wide}

\newcommand{\alq}{\frac{{\mathbf a} + \bm\lambda}{q}}
\newcommand{\Section}[1]{\section{#1} \setcounter{equation}{0}}

\newcommand\R{\mathbb{R}}
\newcommand\Z{\mathbb{Z}}
\newcommand\N{\mathbb{N}}

\newcommand{\vv}[1]{{\mathbf{#1}}}

\newtheorem{thm}{Theorem}

\newtheorem{cor}{Corollary}

\newtheorem{thp}{Proposition}

\def\hs{{\mathcal H}^s}
\def\cS{\mathcal S}
\def\cR{\mathcal R}
\def\cU{\mathcal U}
\def\cM{\mathcal M}

\def\cD{\mathcal D}
\newcommand{\f}{\mathbf{f}}

\begin{document}

\renewcommand{\baselinestretch}{1.2}

\title[Diophantine Approximation on manifolds: the convergence theory]{Diophantine approximation on manifolds and the distribution of rational points: contributions \\ to the convergence theory}

\author{ V. Beresnevich, R.C. Vaughan,  S. Velani}
\address{RCV: Department of Mathematics, McAllister Building,
Pennsylvania State University, University Park, PA 16802-6401,
U.S.A.} \email{rvaughan@math.psu.edu}
\author{E. Zorin   }
\address{VB, SV, EZ: Department of Mathematics, University of York, Heslington, York, YO10
5DD, U.K.} \email{victor.beresnevich@york.ac.uk, sanju.velani@york.ac.uk, evgeniy.zorin@york.ac.uk }

\date{\today}

\thanks{VB and SV:  Research supported in part by EPSRC Programme Grant: EP/J018260/1.}

\thanks{RCV: Research supported in part by NSA grant number H98230-09-1-0015 and H98230-06-0077.}

\thanks{EZ:  Research supported in part by EPSRC First Grant:  	EP/M021858/1.}

\maketitle

\vspace*{-3ex}

\begin{center}
\small {\em 
In memory of Klaus Roth (29 October 1925 -- 10 November 2015)}
\end{center}

\begin{abstract}
In this paper we develop the convergence theory of simultaneous, inhomogeneous Diophantine approximation on manifolds. A consequence of our main result is that if the manifold $\cM \subset \R^n$ is of dimension strictly greater than $(n+1)/2$ and satisfies a natural non-degeneracy condition, then $\cM$ is of Khintchine type  for convergence.
The key lies in  obtaining  essentially the best possible upper bound regarding the distribution of rational points near manifolds.
\end{abstract}

{\footnotesize

 }

\Section{Introduction and statement of results \label{convintro}}

\subsection{The setup\label{setup}} Throughout, we suppose that $m\le d$, $n=m+d$ and that $\mathbf f=(f_1,\ldots,f_m)$ is defined on $\mathcal U=[0,1]^d$.  Suppose further that $\partial\mathbf f/\partial\alpha_i$ and $\partial^2\mathbf f/\partial\alpha_i\partial\alpha_j$ exist and are continuous on $\mathcal U$, and that there is an $\eta>0$ such that for all $\bm\alpha\in\cU$
\begin{equation}
\label{e:one1} \left|
\det\left(
\frac{\partial^2 f_j}{\partial\alpha_1\partial\alpha_i}(\bm\alpha)
\right)_{\substack{
1\le i\le m\\
1\le j\le m
}}
\right|\ge\eta.
\end{equation}

\noindent Throughout $\R^+=[0,+\infty)$ is the set of non-negative real numbers. Let $\psi: \R^+ \to \R^+ $ be a function such that
 $ \psi(r)\rightarrow0$ as $r\rightarrow\infty$ and $\bm\theta=(\bm\lambda,\bm\gamma) \in \R^d \times \R^m$.  Now for a fixed $ q \in \N $,   consider the set
 \begin{equation}
\label{sv1}
\mathcal R(q,\psi,\bm\theta):= \left\{(\mathbf a, \mathbf b)  \in \Z^d \times \Z^m    :
  \begin{array}{l}
  \alq \in  \mathcal U \, ,  \\[1ex]
    |q \mathbf f\big(\alq\big) - \bm\gamma - {\mathbf b} | < \psi(q)
    \end{array}
\right\}
\end{equation}
and let
$$ A(q,\psi,\bm\theta) :=  \# \mathcal R(q,\psi,\bm\theta) \, .
$$

\noindent  The map $\mathbf f  : \mathcal U  \to \R^m$ naturally gives rise to the $d$-dimensional  manifold
\begin{equation}
\label{monge}
 \cM_{\mathbf f} := \left\{ (\alpha_1, \ldots, \alpha_d, f_1(\boldsymbol{\alpha} ), \ldots, f_m(\boldsymbol{\alpha})) \in \R^n : \boldsymbol{\alpha} =(\alpha_1, \ldots, \alpha_d) \in \mathcal U  \right\}  \,
\end{equation}
embedded in $\R^n$.  Recall that  by the  Implicit  Function Theorem  any smooth manifold $\cM$  can be locally defined in  this manner; i.e. with a Monge parametrisation. The upshot is that,  $ A(q,\psi,\bm\theta)$   counts the number of shifted rational points
$$
\textstyle{\big( \frac{a_1+\lambda_1}{q}, \dots, \frac{a_d+\lambda_d}{q}, \frac{b_1+\gamma_1}{q}, \ldots, \frac{b_m+\gamma_m}{q} \big) }  \in \R^n
$$
that  lie (up to an absolute constant) within the  $\psi(q)/q$-neighbourhood of     $\cM_{\mathbf f}$.  Before stating our counting results it is worthwhile to  compare condition (\ref{e:one1})  imposed  on the Jacobian of $\mathbf f$ with that of non-degeneracy as defined by  Kleinbock and Margulis in their pioneering work \cite{Kleinbock-Margulis-98:MR1652916}.  In this paper they  prove the  Baker-Sprind\v{z}uk `extremality' conjecture in the theory of  Diophantine approximation on manifolds.

The above map  $\mathbf f  : \mathcal U  \to \R^m   : \boldsymbol{\alpha} \mapsto \f(\boldsymbol{\alpha})=(f_1(\boldsymbol{\alpha}),\dots,f_m(\boldsymbol{\alpha}))$   is said to be \emph{$l$-non--degenerate at} $\boldsymbol{\alpha}\in \mathcal U$ if there exists some integer  $l \geq 2 $ such that $\f$  is $l$ times continuously differentiable on some
sufficiently small ball centred at $\boldsymbol{\alpha}$ and the partial derivatives
of $\f$ at $\boldsymbol{\alpha}$ of orders $2$ to $l$ span $\R^m$. The map $\f$ is
called \emph{non--degenerate} if it is $l$-non--degenerate at almost every (in
terms of $d$--dimensional Lebesgue measure) point in $\cU$; in turn
the manifold $ \cM_{\mathbf f}$ is also said to be non--degenerate. Non-degenerate manifolds  are smooth
sub-manifolds of $\R^n$ which are sufficiently curved so as to
deviate from any hyperplane at a polynomial rate see \cite[Lemma~1(c)]{Beresnevich-02:MR1905790}.  As is well known, see \cite[p.~341]{Kleinbock-Margulis-98:MR1652916}, any real connected analytic manifold not contained in any hyperplane of $\R^n$  is non--degenerate.

It follows from the definition of $l$-non-degeneracy, that condition (\ref{e:one1})  imposed  on $\f$   implies  that $\f $ is $2$-non-degenerate at every point. Although (\ref{e:one1})  is fairly generic, the converse is not always true even if we allow rotations of the coordinate system. The submanifold $(x,y,z_1,\dots,z_k,x^2,xy,y^2)$ of $\R^{k+5}$ provides a counterexample\footnote{The authors are grateful to David Simmons for providing the example.}.

\subsection{Results on counting rational points}\label{counting section}
Throughout, the Vinogradov symbols $\ll$ and $\gg$ will be used to indicate an inequality with an unspecified positive multiplicative constant.   If $a \ll b $ and $ a \gg b $, we write $a \asymp b $ and say that the two quantities $a$ and $b$ are comparable. Throughout the paper the constants will only depend on the dimensions $n$ and $d$ and the map~$\vv f$.

Observe  that for $ q$ sufficiently large so that $\psi(q) \leq  1/2 $ , we have that
\begin{equation}\label{A}
A(q,\psi,\bm\theta) =    \# \left\{\mathbf a \in \Z^d   :
  \begin{array}{l}
  \alq \in  \mathcal U \, ,  \\[1ex]
    \|q \mathbf f\big(\alq\big) - \bm\gamma  \| < \psi(q)
    \end{array}
\right\} \,
\end{equation}

where as usual $ \| \vv x \| :=  \max_{1 \le i \le m} \| x_i \|  $ for any  $\vv x  \in \R^m$.
In particular, when $0 < \psi(q) \leq  1/2  $,   the  obvious heuristic argument leads us to the following estimate:
\begin{equation}\label{heur}
 A(q,\psi,\bm\theta)   \ \asymp \  q^n  \left( \frac{\psi(q)}{q}  \right)^m  = \,   \psi(q)^m \, q^d  \, .
\end{equation}

We establish the following  upper bound result.

\begin{thm}
\label{t:svthm1}  Suppose that $\mathbf f  : \mathcal U  \to \R^m $ satisfies  \eqref{e:one1} and $\bm\theta \in \R^n$.   Suppose that  $0 < \psi(q) \leq 1/2  $.     Then
\begin{equation}
\label{sv2} A(q,\psi,\bm\theta)     \, \ll \,   \psi(q)^m \, q^d  \ + \   (q \, \psi(q))^{-1/2} q^d    \,   \max \{1, \log (q \, \psi(q))  \}  \; ,
\end{equation}
where the implied constant is independent of $q$, $\bm\theta$ and $\psi$ but may depend on $\mathbf f$.
\end{thm}

\noindent The following is a straightforward consequence of the theorem.  It states  that the upper bound (\ref{sv2}) coincides with  the heuristic estimate if $\psi(q)$ is not too small.

\begin{cor}
\label{c:svcor1} Suppose that $\mathbf f  : \mathcal U  \to \R^m $ satisfies  \eqref{e:one1} and $\bm\theta \in \R^n$. Suppose that $$ q^{-1/(2m+1)}(\log q)^{2/(2m+1)} \le \psi(q) \leq 1/2  \, . $$    Then for integers $q\ge2$ we have that
\begin{equation}
\label{sv3} A(q,\psi,\bm\theta)     \, \ll \,   \psi(q)^m \, q^d  \; .
\end{equation}
\end{cor}

\subsection{Results on metric Diophantine approximation}\label{KTresults}

Given a function  $\psi: \R^+ \to \R^+ $ and   a point   $\bm\theta =(\theta_1,\ldots,\theta_n) \in \R^n$,  let  $\cS_n(\psi,\bm\theta)$ denote the set of $\vv y=(y_1,\dots,y_n)\in\R^n$ for which  there exists  infinitely many
$q\in\N$ such that $$ \|q\vv y- \bm\theta\|  =  \max_{1\le i\le n}\|q y_i- \theta_i\|<\psi(q) \, .  $$
In the case that the inhomogeneous part $\bm\theta $ is the origin, the corresponding set  $\cS_n(\psi):=\cS_n(\psi,\bm0)$ is the usual homogeneous set of simultaneously $\psi$--approximable points in $\R^n$.
In the case $\psi$ is $\psi_{\tau}:r\to r^{-\tau}$ with $\tau>0$, let us write
 $\cS_n(\tau,\bm\theta)$  for $\cS_n(\psi,\bm\theta)$ and $\cS_n(\tau)$  for $\cS_n(\tau,\bm0)$.   Note that in view
of Dirichlet's theorem ($n$-dimensional simultaneous version),
$\cS_n(\tau)= \R^n $ for any $\tau \leq 1/n$.

In the general discussion above we have not made any assumption on $\psi$ regarding monotonicity.   Thus the integer support of $\psi$ need not be $\N$.   Throughout,   $\mathcal N  \subset \N  $  will denote  the integer support of $\psi$.   That is the set of $q\in\N$ such that $\psi(q)>0 $. Regarding the set  $\cS_n(\psi,\bm\theta)$, measure theoretically, this is equivalent  to saying that we are only interested in integers $q$ lying in some  given set $\mathcal N $ such as the set of primes or squares or powers of two.   The theory of restricted Diophantine approximation in $\R^n$  is both topical and well developed for certain sets $\mathcal N $  of number theoretic interest -- we refer the reader to \cite[Chp 6]{Harman-1998a} and \cite[\S12.5]{Beresnevich-Dickinson-Velani-06:MR2184760} for further details.  However, the  theory of restricted Diophantine approximation on manifolds is not so well developed.

Armed with Corollary \ref{c:svcor1}, we are able to establish the following convergent statement for the $s$-dimensional Hausdorff measure $\hs$ of $\cM_{\mathbf f} \cap \cS_n(\psi,\bm\theta)$.  Note that if $ s >  d=\dim \cM_{\mathbf f}   $, then   $\hs \left( \cM_\mathbf f \cap \cS_n(\psi,\bm\theta) \right) = 0$  irrespective of $\psi$. This follows immediately from the definition of Hausdorff dimension and that fact that
$$
\dim (\cM_{\mathbf f}\cap \cS_n(\psi,\bm\theta))  \, \leq  \, \dim \cM_{\mathbf f}  \, .
$$

\begin{thm}\label{t:genconv}
Let $\bm\theta \in \R^n$ and  $\psi: \R^+ \to \R^+ $ be a  function such that
$ \psi(r)\rightarrow0$ as $r\rightarrow\infty$ and
\begin{equation}\label{condpsi}
\psi(q)\ge q^{-1/(2m+1)}(\log q)^{2/(2m+1)}\qquad\text{for all $q \in \mathcal N $,}
\end{equation}
where as $\mathcal{N}=\{q\in\N:\psi(q)>0\}$.
Let $0< s \le d   $  and $\mathbf f : \mathcal U \to  \R^m$   satisfy the following condition
\begin{equation}\label{ndg}
 \hs \big( \big\{ \bm\alpha \in \mathcal U:  \text{the l.h.s. of \eqref{e:one1}}    = 0 \big\}\big)  = 0.
\end{equation}
Then
\begin{equation*}
\hs \big( \cM_{\mathbf f}\cap \cS_n(\psi,\bm\theta) \big) = 0   \qquad \text{whenever} \qquad  \sum_{q=1}^{\infty}  \left( \textstyle{\frac{\psi(q)}{q}}  \right)^{s+m} q^n   \ < \ \infty \, .
\end{equation*}
\end{thm}

\bigskip

\noindent\emph{Remark}~1.
Recall, that in view of the discussion in \S\ref{setup} the condition imposed on $\f$ in the above theorem and its corollaries below are equivalent to saying that the manifold is $2$-non-degenerate everywhere except on a set of Hausdorff $s$-measure zero.

Now we consider two special cases of Theorem~\ref{t:genconv}. First suppose the integer support of $\psi$ is along a lacunary sequence.  In particular,  consider the concrete situation that $  \mathcal N =\{2^t: t \in \N \}$.  The following statement is valid for any $ n = d+m $  and to the best of our knowledge is first result of its type even within the setup of planar curves ($d=m=1$).

\begin{cor}\label{c:svlacunary}
Let $\bm\theta \in \R^n$ and $\psi: \R^+ \to \R^+ $ be a  function such that
 $ \psi(r)\rightarrow0$ as $r\rightarrow\infty$  and $  \mathcal N =\{2^t: t \in \N \}$.
  Let  $$
 d -  \textstyle{\frac{n}{2(m+1)}}< s \le d$$
 and assume that  $\mathbf f : \mathcal U \to  \R^m$ satisfies \eqref{ndg}.
 Then
 \begin{equation*}
\hs \left( \cM_{\mathbf f}\cap \cS_n(\psi,\bm\theta) \right) = 0   \qquad {  if } \qquad  \sum_{t=1}^{\infty}  \left( 2^{-t} \,  \psi(2^t)\right)^{s+m}  2^{t n}  \ < \ \infty \, .
\end{equation*}
\end{cor}

\begin{proof}
Consider the auxiliary function
$$
\tilde\psi(q)=\max\{\psi(q),Cq^{-1/(2m+1)}(\log q)^{2/(2m+1)}\}\,,
$$
where $C>0$ is a sufficiently large constant. Then as is easily verified using the convergence sum condition of Corollary~\ref{c:svlacunary}
$$
\sum_{t=1}^{\infty}  \left( 2^{-t} \,  \tilde\psi(2^t)\right)^{s+m}  2^{t n}  <\infty
$$
and therefore, by Theorem~\ref{t:genconv}, we have that $\hs \left( \cM_{\mathbf f}\cap \cS_n(\tilde\psi,\bm\theta) \right) = 0$. Trivially, we have that
$\cS_n(\psi,\bm\theta) \subset \cS_n(\tilde\psi,\bm\theta)$ and then the required statement follows on using the monotonicity of $\hs$.
\end{proof}

Note that \eqref{ndg} is always satisfied  if $\dim ( \left\{ \boldsymbol{\alpha} \in \mathcal U:  \  \text{the l.h.s. of \eqref{e:one1}}    = 0 \right\}) \leq d -  \frac{n}{2(m+1)}$.

Let us now consider Theorem \ref{t:genconv} under the assumption  that $\psi $ is monotonic.   Then, without loss of generality, we can assume that  $\mathcal N = \N$  since otherwise $\psi(q) = 0$ for all   sufficiently large $q$  and so  $\cS_n(\psi,\bm\theta)$ is the empty set and there is nothing to prove.  Furthermore,  we can assume that $\psi(q) \ll q^{-1/n}$ for all $q\in\N$ since otherwise the $s$-volume sum appearing in the theorem is divergent for $s \le d$. This is in line with the fact that if $\psi(q) \ge  q^{-1/n}$  for all   sufficiently large   $q$, then  by Dirichlet's theorem we have that
$
 \cM_{\mathbf f}\cap \cS_n(\psi)  =  \cM_{\f}
$
and so $\hs \left( \cM_{\mathbf f}\cap \cS_n(\psi) \right)  >0 $ for $s \le d$.  The upshot is that within the context of Theorem \ref{t:genconv},  for monotonic $\psi$  we can assume that
$$
q^{-1/(2m+1)}(\log q)^{2/(2m+1)} \ll \psi (q)  < q^{-1/n} \, .
$$
This forces $d > (n+1)/2 $. 


\begin{cor}
\label{c:svmontonic}
 Let $\bm\theta \in \R^n$ and $\psi: \R^+ \to \R^+ $ be a  monotonic function such that
 $ \psi(r)\rightarrow0$ as $r\rightarrow\infty$.   Let
$$  d   > \textstyle{\frac{n+1}{2}} \,     \quad  \text{and}   \quad  s_0:=
 \textstyle{\frac{dm}{m+1}} + \textstyle{\frac{n+1}{2(m+1)}}< s \le d  \,  $$
and assume that $\mathbf f : \mathcal U \to  \R^m$   satisfies \eqref{ndg}. Then
 \begin{equation*}
\hs \left( \cM_{\mathbf f}\cap \cS_n(\psi,\bm\theta) \right) = 0   \qquad \text{whenever} \qquad  \sum_{q=1}^{\infty}  \left( {\frac{\psi(q)}{q}}  \right)^{s+m} q^n   \ < \ \infty \, .
\end{equation*}
 \end{cor}

 \noindent The proof is similar to that of Corollary~\ref{c:svlacunary}. Note that \eqref{ndg}  is always satisfied  if $$\dim ( \left\{ \boldsymbol{\alpha} \in \mathcal U:  \  {\rm l.h.s. \ of \ } (\ref{e:one1})    = 0 \right\}) \leq s_0.$$ Also note that  the condition $d > (n+1)/2 $  guarantees that $s_0 < d $.  However, it does mean that the  corollary is not applicable  when $n= 3$ or $n=2$. The fact that is not applicable when $n=2$ is not a concern - see Remark~2 immediately  below.

\noindent\emph{Remark}~2.
It is conjectured  that the conclusion of Corollary \ref{c:svmontonic} is valid for any non-degenerate manifold  (i.e. $d \geq 1 $) and  $  \textstyle{\frac{dm}{(m+1)}} < s \le d$ -- see for example \cite[\S8]{Beresnevich-SDA1}.  For planar curves ($d=m=1$), this is known to be true \cite{Beresnevich-Vaughan-Velani-11:MR2777039,Vaughan-Velani-06:MR2242634}. To the best of our knowledge, beyond planar curves, the corollary represents the first significant contribution in favour of the  conjecture.

\noindent\emph{Remark}~3.
Corollary \ref{c:svmontonic} together with  the definition of Hausdorff dimension implies that if $d > (n+1)/2 $, then for $1/n \leq \tau \leq  1/(2m+1)$
$$
\dim \left( \cM_{\mathbf f}\cap \cS_n(\tau,\bm\theta) \right)   \, \leq \, \textstyle{\frac{n+1}{\tau +1}}   - m \, .
$$

\noindent\emph{Remark}~4.
Corollary \ref{c:svmontonic} with $s=d$ implies that if $d > (n+1)/2 $ then
\begin{equation} \label{lebcon}
| \cM_{\mathbf f}\cap \cS_n(\psi,\bm\theta) |_{\cM_{\mathbf f}}  = 0   \qquad \text{whenever} \qquad  \sum_{q=1}^{\infty}  \psi(q)^n \ < \ \infty \, ,
\end{equation}
where $ |  \ . \  |_{\cM_{\mathbf f}} $   is the induced $d$-dimensional Lebesgue measure on $\cM_{\mathbf f}$.  In other words, it proves that  the $2$-non-degenerate submanifold $\cM_{\mathbf f}$ of $\R^n$  with dimension strictly greater than  $(n+1)/2$ is of  Khintchine-type for convergence -- see \cite{Beresnevich-Dickinson-Velani-07:MR2373145}.  Apart from the  planar curve results referred to in Remark~2,  the current state of the convergent Khintchine theory is somewhat ad-hoc. Either a specific
manifold or a special class of manifolds satisfying various constraints is studied. For example it has been shown that (i) manifolds which are a topological product of at least four non--degenerate planar curves are Khintchine type for convergence \cite{Ber73} as are (ii) the so called 2--convex manifolds of dimension
$d \geq 2$ \cite{DRV91} and (iii) straight lines through the origin satisfying a natural Diophantine condition \cite{Kov00}.

\noindent\emph{Remark}~5.
In view of the conjecture mentioned above in Remark~2,   we expect  (\ref{lebcon}) to remain  valid for any non-degenerate manifold  without any restriction on its dimension. Note that it is relatively straightforward to establish  that this is indeed the case   for almost all $\bm\theta$. Moreover, we do not need to assume that $\psi$ is monotonic or even that $\cM_{\mathbf f}$  is non-degenerate.  In other words,  for any $C^1$ submanifold\footnote{By a $C^1$ submanifold we mean an immersed manifold into $\R^n$ by a $C^1$ map, that is the image of a $C^1$ map $\vv f:\cU\to\R^n$.}  $\cM_{\mathbf f}$ of $\R^n$ and  $\psi: \R^+ \to \R^+ $,  we have that  (\ref{lebcon}) is valid for almost all $\bm\theta \in \R^n$.  This is an immediate  consequence of the following even more general `doubly metric' result.

\begin{thp}
Let $\vv f:\cU\to\R^n$ be any continuous map. Given  $\psi: \R^+ \to \R^+ $, let
$$\cD(\vv f,\psi) := \{ (\vv x,\bm\theta) \in \cU\times \R^n:     \|q\vv f(\vv x)- \bm\theta\|  <\psi(q) \text{ for i.m. } q\in\N \}  $$
and let $|  \ . \  |_{d+n } $ denote $(d+n)$-dimensional Lebesgue measure.
Then
\begin{equation} \label{dob}
| \cD(\vv f,\psi) |_{d+n}   = 0   \qquad \text{whenever} \qquad  \sum_{q=1}^{\infty}  \psi(q)^n \ < \ \infty \, .
\end{equation}
\end{thp}

\noindent {\em Proof.}  The proposition is pretty much a direct  consequence of Fubini's theorem.  Without loss of generality, we can assume that $\bm\theta$ is restricted to the unit cube $[0,1]^n$.   For $q \in \N$,  let
$$
\delta_q(\vv x) :=  \left\{\begin{array}{ll}
1  & \text{if} \;\;\;
\|\vv x  \| < \psi(q) \\[1ex]
0 &
\textstyle\text{otherwise }  \;
\end{array}\right.
$$
and
$$
D_q(\vv f,\psi) := \{ (\vv x,\bm\theta) \in \cU\times [0,1]^n:    \delta_q(q\vv f(\vv x)- \bm\theta) =1  \}  \, .
$$
Notice that
$$
\cD(\vv f,\psi) =  \limsup_{q \to \infty}  D_q(\vv f,\psi) \, ,
$$
and that by Fubini's theorem
\begin{eqnarray*}
| D_q(\vv f,\psi)  |_{d+n}  & = & \int_{\cU}  \Big(  \int_{[0,1]^n} \delta_q(q\vv f(\vv x)- \bm\theta) d\bm\theta  \Big)  \, d\vv x \\[2ex]
& = &  | \cU |_{d} \; (2 \, \psi(q))^n = (2 \, \psi(q))^n  \, .
\end{eqnarray*}

Hence
$$
\sum_{q=1}^{\infty} | D_q(\vv f,\psi)  |_{d+n}  \ \  \asymp   \ \ \sum_{q=1}^{\infty}  \psi(q)^n \ < \ \infty \, ,
$$
and the Borel-Cantelli lemma implies the desired measure zero statement.  \hfill $\Box$

\subsection{Restricting to hypersurfaces}
As already mentioned,   the condition $d > (n+1)/2 $  means that Corollary \ref{c:svmontonic} is not applicable when $n=3$. We now attempt to rectify this.  In the case $m=1$, so that the manifold $\cM_{\mathbf f}  $ associated with $\mathbf f$  is a hypersurface,  we can do better than Theorem~\ref{t:svthm1} if we  assume that $\cM_{\mathbf f} $ is  genuinely  curved.  More precisely,  in place of (\ref{e:one1}) we suppose that there is an $\eta>0$ such that for all $\bm\alpha\in\cU$
\begin{equation}
\label{e:one3} \left|
\det\left(
\frac{\partial^2 f}{\partial\alpha_i\partial\alpha_j}(\bm\alpha)
\right)_{\substack{
1\le i\le d\\
1\le j\le d
}}
\right|\ge\eta
\end{equation}
 where for brevity we have written $f$ for $f_1$.  It is not too difficult to  see that this  condition  imposed on the determinant (Hessian) is valid for spheres but not for cylinders with a flat base. We will refer to the hypersurface $\cM_{\mathbf f} $ with $\mathbf f$  satisfying (\ref{e:one3}) as \emph{genuinely curved}. Throughout the rest of this section we will assume that $m=1$ and so $d=n-1$.

\begin{thm}
\label{t:svthm3} Suppose that  $\vv f  : \cU  \to \R $  satisfies  \eqref{e:one3} and $\bm\theta \in \R^n$.  Suppose that  $0 < \psi(q) \leq 1/2  $.      Then
\begin{equation}
\label{sv22} A(q,\psi,\bm\theta)     \ \ll \   \psi(q) \, q^d  \, + \,   (q \, \psi(q))^{-d/2} q^d   \,   \max \{1, (\log (q \, \psi(q)))^d \}   \;
\end{equation}
where the implied constant is independent of $q$, $\bm\theta$ and $\psi$ but may depend on $\mathbf f$.
\end{thm}

A simple consequence of this theorem is the following   analogue of Corollary \ref{c:svcor1}.

\begin{cor}\label{c:svcor2}
Suppose that $\mathbf f  : \mathcal U  \to \R $ satisfies  \eqref{e:one3} and $\bm\theta \in \R^n$. Suppose that
 $$  q^{-d/(2+d)}(\log q)^{2d/(2+d)}   \le \psi(q) \leq 1/2 \, . $$     Then
 for integers $q\ge2$ we have that
\begin{equation}
\label{sv33} A(q,\psi,\bm\theta)     \, \ll \,   \psi(q)  \, q^d  \; .
\end{equation}
\end{cor}

\noindent It is easily seen that  Theorem~\ref{t:svthm1}  with $m=1$  and Theorem \ref{t:svthm3} coincide when $n = 2$ but for $ n  \geq 3 $ the second term on the r.h.s. in (\ref{sv22})  is smaller than the corresponding term in (\ref{sv2}).  In particular,
$$
q^{-d/(2+d)}(\log q)^{2d/(2+d)}   < q^{-1/3}(\log q)^{2/3}
$$
and so Corollary \ref{c:svcor2} is stronger than Corollary \ref{c:svcor1} for $f$ satisfying  (\ref{e:one3}).  Corollary \ref{c:svcor2} enables us to obtain the  analogue of Theorem~\ref{t:genconv} for genuinely curved hypersurfaces in which the condition that  $\psi(q)\gg q^{-1/(2m+1)}(\log q)^{2/(2m+1)}$ for $q \in \mathcal N $ is replaced by $\psi(q)\gg q^{-d/(2+d)}(\log q)^{2d/(2+d)}$ for $q \in \mathcal N $.  In turn for monotonic functions we have the following statement.  It represents a strengthening of Corollary \ref{c:svmontonic} in the case of   genuinely curved hypersurfaces and  is valid when $n=3$.

\begin{cor}\label{c:svhypermontonic}
Suppose that   $\mathbf f  : \mathcal U  \to \R $ and $\bm\theta \in \R^n$. Let $\psi: \R^+ \to \R^+ $ be a  monotonic function such that
 $ \psi(r)\rightarrow0$ as $r\rightarrow\infty$.    Let
$$ n \geq 3  \,     \qquad  and   \qquad  \textstyle{\frac{n-1}{2}} + \textstyle{\frac{n+1}{2n}}< s \le n-1  \,  $$
and assume that
\begin{equation*}
 \hs \big( \big\{ \bm\alpha \in \mathcal U:  \text{the l.h.s. of \eqref{e:one3}}    = 0 \big\}\big)  = 0.
\end{equation*}
  Then
 \begin{equation*}
\hs \left( \cM_{\mathbf f}\cap \cS_n(\psi,\bm\theta) \right) = 0   \qquad\text{whenever}\qquad  \sum_{q=1}^{\infty}  \left( {\frac{\psi(q)}{q}}  \right)^{s+1} q^n   \ < \ \infty \, .
\end{equation*}
\end{cor}

The conjectured lower bound for $s$\/ above is $(n-1)/2$ -- see Remark~2 preceding the statement of  Corollary \ref{c:svmontonic}. The proof of the above corollary is similar to that of Corollary~\ref{c:svlacunary}.

\subsection{Further remarks and other developments}

The upper bound results of \S\ref{counting section} for the counting function $A(q,\psi,\bm\theta)$  are at the heart of establishing the convergence results of \S\ref{KTresults}.  We emphasize that $A(q,\psi,\bm\theta)$  is defined for a fixed $q$ and that Theorem  \ref{t:svthm1}  provides an upper bound for this function for any $q$ sufficiently large.  It is this fact,  that enables us to obtain  convergent results such as   Theorem \ref{t:genconv} without assuming that $\psi$ is monotonic. While statements without monotonicity are desirable, considering counting functions for a fixed $q$ does prevent us from taking advantage of any potential averaging over $q$.  More precisely, for $Q > 1$ consider the counting function
\begin{align}
N(Q,\psi,\bm\theta)  \; &  :=   \;  \#  \left\{(q, \mathbf a, \mathbf b)  \in  \N \times  \Z^d \times \Z^m    :
  \begin{array}{l}
  Q < q \leq 2Q, \; \alq \in  \mathcal U \, , \\[1ex]
       |q \mathbf f\big(\alq\big) - \bm\gamma - {\mathbf b} | < \psi(q)
                           \end{array}
\right\} \nonumber  \\[2ex]
\,~~&  =    \sum_{ Q < q \leq 2Q }      A(q,\psi,\bm\theta)  \label{sv99}  \, .
\end{align}

\noindent If $\psi$ is monotonic, then $\psi(q)  \le \psi(Q) $ for $ Q < q \leq 2Q $ and the obvious heuristic `volume' argument leads us to the following estimate:
\begin{equation} \label{nQ}
N(Q,\psi,\bm\theta)   \ \ll \  \psi(Q)^m \, Q^{d+1}  \, .
\end{equation}
Clearly, the upper bound \eqref{sv3} for $A(q,\psi,\bm\theta)$ as obtained in Corollary~\ref{c:svcor1} implies \eqref{nQ}.    The converse is unlikely to be true.  However, for monotonic $\psi$ establishing  \eqref{nQ} suffices to prove convergence results such as  Corollary \ref{c:svmontonic}. Indeed, the  fact that we have a complete convergence theory for planar curves (see Remark~2 in \S\ref{KTresults}) relies on the fact that we are able to establish  \eqref{nQ} with $m=1=d$.   Note that the counting result obtained in this paper for  $A(q,\psi,\bm\theta)$ is not strong enough to imply any sort of convergent Khintchine type result for planar curves with $\psi$ monotonic. Furthermore, it is worth pointing out that averaging over $q$ when considering $N(Q,\psi,\bm\theta)$ also has the potential to weaken the lower bound condition \eqref{condpsi} on $\psi$ appearing in Theorem~\ref{t:genconv}. This in turn would increase the range of $s$ within Corollaries~\ref{c:svmontonic} and \ref{c:svhypermontonic}.

Regarding lower bounds for the counting function  $N(Q,\psi,\bm\theta)$, if  $\psi$ is monotonic, then
$\psi(q)  \ge \psi(Q) $ for $ \tfrac12Q < q~\leq~Q $ and the  heuristic `volume' argument leads us to the following estimate:
\begin{equation}\label{lb}
N(\tfrac12Q,\psi,\bm\theta)   \ \gg \  \psi(Q)^m \, Q^{d+1}  \, .
\end{equation}
In the homogeneous case (i.e. when $\bm\theta = \bm0$), the lower bound  given by  \eqref{lb} is established  in~\cite{Beresnevich-SDA1} for  any analytic non-degenerate manifold $\cM$  embedded in $\R^n$ and  $\psi$ satisfying  $\lim_{q\to\infty}q\psi(q)^m=\infty$. When $\cM$ is  a curve,  the condition on $\psi$ can be weakened to  $\lim_{q\to\infty}q\psi(q)^{(2n-1)/3}=\infty$.   Moreover, it is shown in~\cite{Beresnevich-SDA1} that the rational points $ {\mathbf a}  /q$ associated with $N(\tfrac12Q,\psi,\bm0) $ are `ubiquitously' distributed  for  analytic non-degenerate manifolds.   This together with the lower bound estimate is very much at the heart of the  divergent Khintchine type results  obtained  in~\cite{Beresnevich-SDA1}  for analytic non-degenerate manifolds.
In a forthcoming  paper \cite{DIV},  we establish the lower bound estimate \eqref{lb} and show that shifted rational points $ \alq$ associated with $N(\tfrac12Q,\psi,\bm\theta) $  are `ubiquitously' distributed for any $C^{n+1}$ non-degenerate curve in $\R^n$ and arbitrary $\bm\theta$. As a consequence, we obtain a divergent Khintchine type theorem for Hausdorff measures. More specifically,
let\/ $\f=(f_1,\dots,f_{n-1}):[0,1]\to\R^{n-1}$ be a $C^{n+1}$ function  such that for almost all $\alpha\in[0,1]  $
\begin{equation}  \label{ohyes} \det\Big(f_j^{(i+1)}(\alpha)\Big)_{1\le i,j\le n-1}   \ \neq  \ 0
   \, . \end{equation}
 Let $\tfrac12<s\le 1$,  $ \bm\theta  \in \R^n$  and $\psi:\R^+\to\R^+$ be a monotonic function such that $\psi(r)\to0$ as $r\to\infty$. It is established in \cite{DIV} that
 \begin{equation*}
\hs \big( \cM_{\mathbf f}\cap \cS_n(\psi,\bm\theta) \big) = \hs(\cM_\f)    \qquad \text{whenever} \qquad  \sum_{q=1}^{\infty}  \left( {\frac{\psi(q)}{q}}  \right)^{s+n-1} q^n   \ = \ \infty \, .
\end{equation*}

\noindent In view of the conditions imposed on $\f$ above,   the associated  manifold  $\cM_{\mathbf f}$  is by definition  a $C^{n+1}$ non-degenerate curve in $\R^n$. When $ s$ is strictly less than one,  non-degeneracy can be replaced by the condition that  \eqref{ohyes} is satisfied for at least one point $\alpha\in[0,1] $.  In other words, all that is required is that there exists at least one point on the curve that is non-degenerate.  Using fibering techniques, it is also shown in  \cite{DIV} that the above statement for non-degenerate curve in $\R^n$ can be  readily extended to accommodate a large class of non-degenerate manifolds  beyond the analytic ones considered in~\cite{Beresnevich-SDA1}.

\section{Preliminaries to the proofs of Theorems \ref{t:svthm1} and \ref{t:svthm3}  \label{prelims} }

To establish Theorems \ref{t:svthm1} and \ref{t:svthm3} we adapt an argument of Sprind\v{z}uk \cite[Chp2~\S6]{Sprindzuk-1979-Metrical-theory}.  In our view the adaptation is non-trivial.

 Suppose  $0 < \psi(q) \leq  1/2  $ and recall that $\bm\theta=(\bm\lambda,\bm\gamma) \in \R^d \times \R^m$. Recall also that $A(q,\psi,\bm\theta)$ is given by \eqref{A}.
Given $\bm\lambda = ( \lambda_1, \ldots, \lambda_d)  \in \R^d$, let   $ \tilde{\bm\lambda} := ( \{\lambda_1\}, \ldots, \{\lambda_d\}) \in [0,1)^d$  denote the  fractional part of  $\bm\lambda$.  Then, it follows that
\begin{equation} \label{fin}
A(q,\psi,\bm\theta) =    \# \mathcal A(q,\psi,\bm\theta)
\end{equation}
where
$$\mathcal A(q,\psi,\bm\theta):=\{
\mathbf a \in \Z(q)  \; : \; \|q \, \mathbf f\big(\tfrac{\mathbf a + \tilde{\bm\lambda}}{q}\big) - \bm\gamma \| < \psi(q) \, \}  \,  $$
and
$$
\Z(q):=   \prod_{i=1}^{d}  \big( [0,q_i ]\cap\mathbb Z \big)    \quad {\rm and   \ } \quad  q_i = \left\{
\begin{array}{l}
  q   \quad \qquad   \! {\rm if  \ }  \tilde{\lambda_i} = 0  \\[1ex]
  q-1  \quad {\rm otherwise.}
\end{array}
\right.
$$

Let $\delta$ be a sufficiently small positive constant that will be determined later and depends on $ \mathbf f$.  Without loss of generality,  we can assume that
$$\delta q\psi(q)>1  \, . $$
Otherwise,   the  error term associated with \eqref{sv2} is, up to a multiplicative constant, larger than the trivial bound
$$
A(q,\psi,\bm\theta)    \leq (q+1)^d
$$
and there is nothing to prove.
Now define
\begin{equation}\label{r}
r:=\lfloor (\delta q\psi(q))^{1/2}\rfloor
\end{equation}
and for each $\mathbf a\in \Z(q)$ write
$$\mathbf a=r\mathbf u(\mathbf a)+\mathbf v(\mathbf a)$$
where $\mathbf u(\mathbf a)$, $\mathbf v(\mathbf a)$ satisfy $u_i(\mathbf a)=\lfloor a_i/r\rfloor$ and $0\le v_i(\mathbf a)<r$  ($ 1 \leq i  \leq  d $).  In particular
$$0\le u_i(\mathbf a)\le s$$
where
$$
s:=\lfloor q/r\rfloor.
$$
For $\mathbf u  \in \Z^d$, define
$$\mathcal A(q, \psi, \bm\theta, \mathbf u):=\{\mathbf a\in\mathcal A(q, \psi,\bm\theta ): \mathbf u(\mathbf a)= \mathbf u\}$$
and
$$A(q, \psi, \bm\theta, \mathbf u):=\# \mathcal A(q, \psi, \bm\theta, \mathbf u).$$

\noindent By the mean value theorem for second derivatives, when $\mathbf a\in\mathcal A(q, \psi, \bm\theta,\mathbf u)$,

$$f_j\big(\tfrac{\mathbf a + \tilde{\bm\lambda}}{q}\big) =
f_j\big(\tfrac{r\mathbf u  + \tilde{\bm\lambda}}{q}\big) +\sum_{i=1}^d \frac{v_i}q \frac{\partial f_j}{\partial\alpha_i}\big(\tfrac{r\mathbf u  + \tilde{\bm\lambda}}{q}\big) +O\left(
\sum_{i=1}^d\sum_{j=1}^d \frac{v_iv_j}{q^2}
\right)$$

\noindent for $\mathbf v=\mathbf v(\mathbf a)\in \mathcal R^d$ where $\mathcal R:=[0,r)\cap\mathbb Z$.  Here the error term is
$$<\ C_1r^2q^{-2} \ \le  \ C_1\delta\psi(q)q^{-1}$$
where $C_1$ depends at most on $d$ and the size of the second derivatives.  Now choose
 $$
\delta  =  1/C_1  \, .
 $$
Thus,  for $\mathbf a=r\mathbf u+\mathbf v$ with $\mathbf a \in \mathcal A(q, \psi, \bm\theta, \mathbf u)$ we have
\begin{equation}
\label{e:two2} \left\|
 qf_j\big(\tfrac{r\mathbf u  + \tilde{\bm\lambda}}{q}\big)+\sum_{i=1}^d v_i \frac{\partial f_j}{\partial\alpha_i}\big(\tfrac{r\mathbf u  + \tilde{\bm\lambda}}{q}\big)  - \gamma_j
 \right\|<2\psi(q)\quad(1\le j\le m).
\end{equation}

\noindent Therefore
$$A(q, \psi, \bm\theta, \mathbf u)\le B(q, \psi, \mathbf u)$$
where $B(q, \psi, \mathbf u):=\#\mathcal B(q, \psi, \mathbf u)$ and
$$\mathcal B(q, \psi, \mathbf u):=\{
\mathbf v\in\mathcal R^d:(\ref{e:two2}) \  {\rm holds} \}.$$

Let
\begin{equation}\label{H}
H:=\left\lfloor\frac1{2\psi(q)}\right\rfloor
\end{equation}
so that $H\ge1$ and $\mathcal H:=[-H,H]\cap\mathbb Z$.  Then
$$\sum_{h\in\mathcal H}\frac{H-|h|}{H^2} \, e(hx)=H^{-2}\left|
\sum_{h=1}^He(hx)
\right|^2=\left(
\frac{\sin \pi H x}{H\sin \pi x}
\right)^2\ge \frac4{\pi^2}$$
whenever $\|x\|\le H^{-1}$.  Thus
$$B(q, \psi, \mathbf u)\ll B^*(q, \psi, \mathbf u)$$
where
\begin{equation}\label{vb1}
B^*(q, \psi, \mathbf u) := \sum_{\mathbf h\in\mathcal H^m} \frac{H-|h_1|}{H^2}\cdots \frac{H-|h_m|}{H^2} \sum_{\mathbf v\in\mathcal R^d} e(\mathbf h.(\mathbf F(\mathbf u,\mathbf v)  - \bm\gamma))
\end{equation}
and
$$
\mathbf h := (h_1,\ldots,h_m) \, ,
$$
$$
\mathbf F:=(F_1,\ldots,F_m) \, ,
$$
$$F_j(\mathbf u,\mathbf v) := qf_j\big(\tfrac{r\mathbf u  + \tilde{\bm\lambda}}{q}\big)+\sum_{i=1}^d v_i \frac{\partial f_j}{\partial\alpha_i}\big(\tfrac{r\mathbf u  + \tilde{\bm\lambda}}{q}\big) \, . $$

\noindent
By the definition of $\mathcal{H}$, we have that
$$
0\le\frac{H-|h_1|}{H^2}\cdots \frac{H-|h_m|}{H^2}\le H^{-m}
$$
for any $\vv h=(h_1,\dots,h_m)\in\mathcal{H}^m$.
Therefore, by \eqref{vb1}, we get that
\begin{equation}\label{vb2}
B^*(q, \psi, \mathbf u) \le H^{-m}\sum_{\mathbf h\in\mathcal H^m} \left|\sum_{\mathbf v\in\mathcal R^d}  e(\mathbf h.(\mathbf F(\mathbf u,\mathbf v)  - \bm\gamma))\right|\,.
\end{equation}
On using the fact that $e(x_1+\dots +x_\ell)=e(x_1)\cdots e(x_\ell)$ and $|e(x)|=1$ for any real numbers $x,x_1,\dots,x_\ell$,  we find  that
\begin{align*}
\left|\sum_{\mathbf v\in\mathcal R^d}  e(\mathbf h.(\mathbf F(\vv u,\vv v)  - \bm\gamma))\right|
&=\left|\sum_{\mathbf v\in\mathcal R^d}  e(\mathbf h.\mathbf F(\vv u,\vv v)) \cdot e( - \mathbf h.\bm\gamma)\right|
=\left|\sum_{\mathbf v\in\mathcal R^d}  e(\mathbf h.\mathbf F(\vv u,\vv v)) \right|\\[2ex]
&=\left|\sum_{\mathbf v\in\mathcal R^d}  e\left(\sum_{j=1}^m h_j\Big(qf_j\big(\tfrac{r\mathbf u  + \tilde{\bm\lambda}}{q}\big)+\sum_{i=1}^d v_i \frac{\partial f_j}{\partial\alpha_i}\big(\tfrac{r\mathbf u  + \tilde{\bm\lambda}}{q}\big)\Big)\right)\right|\\[2ex]
&=\left|\sum_{\mathbf v\in\mathcal R^d}  e\left(\sum_{j=1}^m \sum_{i=1}^d h_jv_i \frac{\partial f_j}{\partial\alpha_i}\big(\tfrac{r\mathbf u  + \tilde{\bm\lambda}}{q}\big)\right)\right|\\[2ex]
&=\left|\sum_{\mathbf v\in\mathcal R^d}  \prod_{i=1}^de\left(v_i\sum_{j=1}^m h_j \frac{\partial f_j}{\partial\alpha_i}\big(\tfrac{r\mathbf u  + \tilde{\bm\lambda}}{q}\big)\right)\right|\\[2ex]
&=\left|\prod_{i=1}^d\sum_{v_i\in\mathcal R}  e\left(v_i\sum_{j=1}^m h_j \frac{\partial f_j}{\partial\alpha_i}\big(\tfrac{r\mathbf u  + \tilde{\bm\lambda}}{q}\big)\right)\right|\,.
\end{align*}
Hence
\begin{align*}
\left|\sum_{\mathbf v\in\mathcal R^d} e(\mathbf h.(\mathbf F(\mathbf u,\mathbf v)  - \bm\gamma))\right| 
&= \prod_{i=1}^d \left|\sum_{v\in\mathcal R}e\left(v\sum_{j=1}^m h_j \frac{\partial f_j}{\partial\alpha_i}\big(\tfrac{r\mathbf u  + \tilde{\bm\lambda}}{q}\big)\right)\right|\,.
\end{align*}
Therefore, by \eqref{vb2}, it follows  that
\begin{equation}\label{vb3}
B^*(q, \psi, \mathbf u)  \ \le \  \frac1{H^m} \sum_{\mathbf h\in\mathcal H^m}\prod_{i=1}^d \left|
\sum_{v\in\mathcal R} e\left(
v\sum_{j=1}^m h_j \frac{\partial f_j}{\partial\alpha_i}\big(\tfrac{r\mathbf u  + \tilde{\bm\lambda}}{q}\big)
\right)
\right|.
\end{equation}

Since $\mathcal{R}=[0,r)\cap\Z$, for any given $\rho\in\R$ we have that $\left|\sum_{v\in\cR}e(v\rho)\right|\le r$ and also that
$$
\left|\sum_{v\in\cR}e(v\rho)\right|=\left|\frac{e(r\rho)-1}{e(\rho)-1}\right|\le\frac{2}{|e(\rho)-1|}\ll\|\rho\|^{-1}\,,
$$
where the implied constant is absolute. Hence, on  taking $\rho=\sum_{j=1}^m h_j \frac{\partial f_j}{\partial\alpha_i}\big(\tfrac{r\mathbf u  + \tilde{\bm\lambda}}{q}\big)$ we have  that
$$
\left|\sum_{v\in\cR}e\left(v\sum_{j=1}^m h_j \frac{\partial f_j}{\partial\alpha_i}\big(\tfrac{r\mathbf u  + \tilde{\bm\lambda}}{q}\big)\right)\right|\ll \min\left(r,\left\|
\sum_{j=1}^m h_j \frac{\partial f_j}{\partial\alpha_i}\big(\tfrac{r\mathbf u  + \tilde{\bm\lambda}}{q}\big)
\right\|^{-1}\right)\,.
$$
This together with \eqref{vb3},  implies that
\begin{equation}\label{vb4}
B^*(q, \psi, \mathbf u)  \ \le \  \frac1{H^m} \sum_{\mathbf h\in\mathcal H^m}\prod_{i=1}^d \min\left(r,\left\|
\sum_{j=1}^m h_j \frac{\partial f_j}{\partial\alpha_i}\big(\tfrac{r\mathbf u  + \tilde{\bm\lambda}}{q}\big)
\right\|^{-1}\right).
\end{equation}

\noindent For a given $\mathbf u\in[0,s]^d$ we consider the intervals $I_i=[u_i-1/2,u_i+1/2]$, unless $u_i=0$ or $u_i=s$ in which case we consider $[u_i,u_i+1/2]$ or $[u_i-1/2,u_i]$ respectively.  For $\beta_i\in I_i$ we have
$$\frac{\partial f_j}{\partial\alpha_i}\big(\tfrac{r\mathbf u  + \tilde{\bm\lambda}}{q}\big) = \frac{\partial f_j}{\partial\alpha_i}\big(\tfrac{r\boldsymbol{\beta} + \tilde{\bm\lambda}}{q}\big) + O(r/q)$$
by the mean value theorem.  Hence
$$\sum_{j=1}^m h_j \left(
\frac{\partial f_j}{\partial\alpha_i}\big(\tfrac{r\mathbf u  + \tilde{\bm\lambda}}{q}\big) - \frac{\partial f_j}{\partial\alpha_i}\big(\tfrac{r\boldsymbol{\beta} + \tilde{\bm\lambda}}{q}\big)
\right) \ll Hr/q$$
where the implicit constant depends at most on $m$ and the size of the second derivatives.  Moreover
$$\frac{Hr^2}q\le \frac{\delta q\psi(q)}{2q\psi(q)} = \frac{\delta}2 <\delta\,,$$
where the left hand side inequality follows from the definitions of $r$ and $H$ -- see \eqref{r} and \eqref{H}.
Hence
$$\left|
\left\|
\sum_{j=1}^m h_j \frac{\partial f_j}{\partial\alpha_i}\big(\tfrac{r\mathbf u  + \tilde{\bm\lambda}}{q}\big)
\right\| -\left\|
\sum_{j=1}^mh_j \frac{\partial f_j}{\partial\alpha_i}\big(\tfrac{r\boldsymbol{\beta} + \tilde{\bm\lambda}}{q}\big)
\right\|
\right| \ll \frac{\delta}r\ll\frac{1}{r}.$$
Thus
$$\min\left(
r, \left\|
\sum_{j=1}^m h_j \frac{\partial f_j}{\partial\alpha_i}\big(\tfrac{r\mathbf u  + \tilde{\bm\lambda}}{q}\big)
\right\|^{-1}
\right)\ll \min\left(
r, \left\|
\sum_{j=1}^m h_j \frac{\partial f_j}{\partial\alpha_i}\big(\tfrac{r\boldsymbol{\beta} + \tilde{\bm\lambda}}{q}\big)
\right\|^{-1}
\right)$$
and furthermore, by considering their product over $i$, we get that
$$\prod_{i=1}^d\min\left(
r, \left\|
\sum_{j=1}^m h_j \frac{\partial f_j}{\partial\alpha_i}\big(\tfrac{r\mathbf u  + \tilde{\bm\lambda}}{q}\big)
\right\|^{-1}
\right)\ll \prod_{i=1}^d\min\left(
r, \left\|
\sum_{j=1}^m h_j \frac{\partial f_j}{\partial\alpha_i}\big(\tfrac{r\boldsymbol{\beta} + \tilde{\bm\lambda}}{q}\big)
\right\|^{-1}
\right).$$
Since the measure of $I_1\times\cdots\times I_d$ is $\asymp1$, integrating the above inequality over $\boldsymbol{\beta}\in I_1\times\cdots\times I_d$ gives that
$$
\prod_{i=1}^d\min\left(
r, \left\|
\sum_{j=1}^m h_j \frac{\partial f_j}{\partial\alpha_i}\big(\tfrac{r\mathbf u  + \tilde{\bm\lambda}}{q}\big)
\right\|^{-1}
\right)\ll \int_{I_1\times\cdots\times I_d}\prod_{i=1}^d\min\left(
r, \left\|
\sum_{j=1}^m h_j \frac{\partial f_j}{\partial\alpha_i}\big(\tfrac{r\boldsymbol{\beta} + \tilde{\bm\lambda}}{q}\big)
\right\|^{-1}
\right)d\boldsymbol{\beta}.
$$
Now recall that the rectangles $I_1\times\cdots\times I_d$ depend on the choice of $\vv u$. Note that their union taken over integer points $\mathbf u\in \mathcal S^d$, where $\mathcal S:=[0,s]$, is exactly $\mathcal{S}^d$. Furthermore, different rectangles can only intersect on the boundary. Hence
summing the above displayed inequality over all integer points $\mathbf u\in \mathcal S^d$ gives
$$
\sum_{\mathbf u\in \mathcal S^d} \prod_{i=1}^d\min\left(
r, \left\|
\sum_{j=1}^m h_j \frac{\partial f_j}{\partial\alpha_i}\big(\tfrac{r\mathbf u  + \tilde{\bm\lambda}}{q}\big)
\right\|^{-1}
\right)\ll \int_{\mathcal{S}^d}\prod_{i=1}^d\min\left(
r, \left\|
\sum_{j=1}^m h_j \frac{\partial f_j}{\partial\alpha_i}\big(\tfrac{r\boldsymbol{\beta} + \tilde{\bm\lambda}}{q}\big)
\right\|^{-1}
\right)d\boldsymbol{\beta}.
$$
Now combining this together with \eqref{vb4} we obtain that
\begin{equation}
\label{e:two3}
\sum_{\mathbf u\in \mathcal S^d} B^*(q, \psi, \mathbf u) \ll H^{-m} \sum_{\mathbf h\in \mathcal H^m} \int_{\mathcal S^d} \prod_{i=1}^d\min\left(
r, \left\|
\sum_{j=1}^m h_j \frac{\partial f_j}{\partial\alpha_i}\big(\tfrac{r\boldsymbol{\beta} + \tilde{\bm\lambda}}{q}\big)
\right\|^{-1}
\right) d\boldsymbol{\beta}.
\end{equation}

\noindent  Now finally observe  that

\begin{equation}
\label{e:svtwo3}
A(q, \psi, \bm\theta) \leq  \sum_{\mathbf u\in \mathcal S^d} A(q, \psi, \bm\theta, \mathbf u)   \, \le \, \sum_{\mathbf u\in \mathcal S^d} B(q, \psi, \mathbf u) \, \ll \,
\sum_{\mathbf u\in \mathcal S^d} B^*(q, \psi, \mathbf u) \,  .
\end{equation}

\Section{The proof of Theorem \ref{t:svthm1}}

\noindent With reference to \S\ref{prelims}, by (\ref{e:two3})
\begin{equation*}
\sum_{\mathbf u\in \mathcal S^d} B^*(q, \psi, \mathbf u) \ll r^{d-1}H^{-m} \sum_{\mathbf h\in \mathcal H^m} \int_{\mathcal S^d} \min\left(
r, \left\|
\sum_{j=1}^m h_j \frac{\partial f_j}{\partial\alpha_1}\big(\tfrac{r\boldsymbol{\beta} + \tilde{\bm\lambda}}{q}\big)
\right\|^{-1}
\right) d\boldsymbol{\beta}.
\end{equation*}
Since (1.1) holds we may make the change of variables
$$\omega_j= \frac{\partial f_j}{\partial\alpha_1}\big(\tfrac{r\boldsymbol{\beta} + \tilde{\bm\lambda}}{q}\big)\quad(1\le j\le m),\, \qquad \omega_j=\beta_j\quad(m<j\le d).$$
Thus
\begin{equation}\label{vb6}
\sum_{\mathbf u\in \mathcal S^d} B^*(q, \psi, \mathbf u) \ \ll  \  \frac{r^{d-1}}{H^m} \sum_{\mathbf h\in \mathcal H^m} \left(
\frac{q}r
\right)^m\int_{\mathcal J_d} \min\left(
r, \left\|
\sum_{j=1}^m h_j \omega_j
\right\|^{-1}
\right) d\boldsymbol{\omega}
\end{equation}
where
$\mathcal J_d:=\mathcal F_1\times\cdots \times\mathcal F_m \times [0,s]^{d-m}$,  $\mathcal F_j:=[f_j^-,f_j^+]$ and
$$f_j^-:=\inf \frac{\partial f_j}{\partial \alpha_1}(\boldsymbol{\alpha})$$
and
$$f_j^+:=\sup \frac{\partial f_j}{\partial \alpha_1}(\boldsymbol{\alpha}).$$
The contribution from $\mathbf h=\mathbf 0$ is
\begin{equation*}
\ll\frac{r^{d-1}}{H^m}\left(\frac qr\right)^m\int_{\mathcal{J}_d}rd\bm\omega\ll
\frac{r^{d-m}}{H^m}\,q^ms^{d-m}
\ll H^{-m}q^d
\end{equation*}
since $rs\asymp q$. Next observe that 
$$
M:=\int_{\mathcal F_1\times\cdots \times\mathcal F_m} \min\left(
r, \left\|
\sum_{j=1}^m h_j \omega_j
\right\|^{-1}
\right) d\omega_1\ldots d\omega_m
$$
is constant with respect to $\omega_{m+1},\dots,\omega_d$. 
Hence, by Fubini's theorem and the fact that $\mathcal J_d:=\mathcal F_1\times\cdots \times\mathcal F_m \times [0,s]^{d-m}$, integrating $M$ over $(\omega_{m+1},\dots,\omega_d)\in[0,s]^{d-m}$ gives that
\begin{equation}\label{vvv}
\int_{\mathcal J_d} \min\left(
r, \left\|
\sum_{j=1}^m h_j \omega_j
\right\|^{-1}
\right) d\boldsymbol{\omega}
= s^{d-m}M\,.
\end{equation}
If $\vv h\neq\vv0$, then assuming for example that $h_1\neq0$ and using Fubini's theorem again we get that
\begin{align*}
M &=\int_{\mathcal F_1\times\cdots \times\mathcal F_m} \min\left(
r, \left\|
\sum_{j=1}^m h_j \omega_j
\right\|^{-1}
\right) d\omega_1\ldots d\omega_m\\[2ex]
&\ll 
\sup_{\rho\in[0,1]}\int_{\mathcal F_1} \min\left(
r, \left\|
h_1 \omega_1-\rho
\right\|^{-1}
\right) d\omega_1\nonumber\\[2ex]
&\ll \sup_{\rho\in[0,1]}~~\sum_{\substack{p\in\Z\\[0.5ex] |p|\ll h_1}}\int_{\mathcal F_1} \min\left(
r, |h_1 \omega_1-\rho-p|^{-1}\right) d\omega_1\nonumber\\[2ex]
&\ll \sup_{\rho\in[0,1]}~~\sum_{\substack{p\in\Z\\[0.5ex] |p|\ll h_1}}\left(\frac{1}{h_1r}+\frac{1}{h_1}\log r\right)\nonumber\\[2ex]
&\ll
\max\{1,\log r\}\,.
\end{align*}
Hence, by the above inequalities and \eqref{vvv}, the contribution from the $\vv h\neq\vv0$ terms within \eqref{vb6} is estimated by
\begin{align*}
&\frac{r^{d-1}}{H^m} \sum_{\mathbf h\in \mathcal H^m} \left(
\frac{q}r
\right)^m s^{d-m}\max\{1,\log r\}\nonumber\\[2ex]
&\ll r^{-1}(rs)^{d-m}q^m\max\{1,\log r\}\nonumber\\[2ex]
&\ll r^{-1}q^d\max\{1,\log r\}.
\end{align*}
In view of (\ref{e:svtwo3}), it follows that
$$
A(q, \psi, \bm\theta)  \ \ll \ H^{-m}q^d   \, + \, r^{-1}q^d\max\{1,\log r\} \ .
$$
Given the definitions of $H$ and $r$ this gives (\ref{sv2}) and thereby completes the proof of the theorem.

\Section{The proof of Theorem \ref{t:svthm3}}

\noindent
Recall that within Theorem~\ref{t:svthm3} we have that $m=1$ and $d=n-1$. Hence, with reference to \S\ref{prelims}, (\ref{e:two3}) becomes
\begin{equation*}
\sum_{\mathbf u\in \mathcal S^d} B^*(q, \psi, \mathbf u) \ll H^{-1} \sum_{h\in \mathcal H} \int_{\mathcal S^d} \prod_{i=1}^d\min\left(
r, \left\|
h \frac{\partial f}{\partial\alpha_i}\big(\tfrac{r\boldsymbol{\beta} + \tilde{\bm\lambda}}{q}\big)
\right\|^{-1}
\right) d\boldsymbol{\beta}\,,
\end{equation*}
where $\vv f=f:\cU  \to \R$. Since  (\ref{e:one3}) holds we may make the change of variables

$$\omega_i= \frac{\partial f}{\partial\alpha_i}\big(\tfrac{r\boldsymbol{\beta} + \tilde{\bm\lambda}}{q})\quad(1\le i\le d).$$
Thus
$$
\sum_{\mathbf u\in \mathcal S^d} B^*(q, \psi, \mathbf u) \ll H^{-1} \sum_{h\in \mathcal H} \left(
\frac{q}r
\right)^d\int_{\mathcal J_d} \prod_{i=1}^d\min\left(
r, \left\|
h \omega_i
\right\|^{-1}
\right) d\boldsymbol{\omega}$$
where
$\mathcal J_d:=\mathcal F_1\times\cdots\times\mathcal F_d$,  $\mathcal F_i:=[f_i^-,f_i^+]$ and
$$f_i^-:=\inf \frac{\partial f}{\partial \alpha_i}(\boldsymbol{\alpha})$$
and
$$f_i^+:=\sup \frac{\partial f}{\partial \alpha_i}(\boldsymbol{\alpha}).$$
The contribution from $h=0$ is
$$
\ll H^{-1}  \left(
\frac{q}r
\right)^d\int_{\mathcal J_d} r^d d\boldsymbol{\omega}
\ll H^{-1}q^d$$
and the contribution from the remaining terms is
\begin{align*}
&\ll H^{-1} \sum_{h\in \mathcal H\setminus\{0\}} \left(
\frac{q}r\right)^d\int_{\mathcal J_d} \prod_{i=1}^d\min\left(r, \left\|h \omega_i\right\|^{-1}\right) d\boldsymbol{\omega}\\[2ex]
&= H^{-1} \sum_{h\in \mathcal H\setminus\{0\}} \left(
\frac{q}r\right)^d \prod_{i=1}^d\int_{\mathcal F_i}\min\left(r, \left\|h \omega_i\right\|^{-1}\right) d\omega_i\\[2ex]
&\ll H^{-1} \sum_{h\in \mathcal H\setminus\{0\}} \left(
\frac{q}r\right)^d \prod_{i=1}^d\max\{1,\log r\}\\[2ex]
&\ll r^{-d}q^d\max\{1,(\log r)^d\}.
\end{align*}
In view of (\ref{e:svtwo3}), it follows that
$$
A(q, \psi, \bm\theta)  \ \ll \ H^{-1}q^d   \, + \, r^{-d}q^d \max\{1,(\log r)^d\} \ .
$$
Given the definitions of $H$ and $r$ this gives (\ref{sv22}) and thereby completes the proof of the theorem.

\Section{Proof of Theorem \ref{t:genconv}}

\noindent{\em Step 1. \ } As mentioned in \S\ref{convintro},   in view of the  Implicit  Function Theorem, we can assume without loss of generality that the   manifold $\cM_{\mathbf f}$  is of the Monge form (\ref{monge}). Note that, since $\cU$ is compact and $\vv f$ is $C^1$, this implies via the Mean Value Theorem that $\mathbf f=(f_1,\ldots,f_m)$ is bi-Lipschitz and so there exists a constant  $ c_1  \ge 1$    such that
\begin{equation} \label{bi}
 \max_{1\le i\le m} |f_i(\boldsymbol{\alpha}) - f_i  (\boldsymbol{\alpha}')|  \,  \le   \, c_1   \;  |\boldsymbol{\alpha}   \, -  \,  \boldsymbol{\alpha}'|  \qquad \forall   \quad \boldsymbol{\alpha}  , \boldsymbol{\alpha}'  \in   \mathcal U=[0,1]^d  \; .
\end{equation}
Let  $\Omega_n^{\mathbf f}(\psi,\bm\theta)$ denote the projection of   $ \cM_{\mathbf f}\cap \cS_n(\psi,\bm\theta) $  onto $ \mathcal U$; that is
$$
\Omega_n^{\mathbf f}(\psi,\bm\theta) := \{\boldsymbol{\alpha} \in \mathcal U  \, : \, (\boldsymbol{\alpha},\mathbf f(\boldsymbol{\alpha})) \in  \cS_n(\psi,\bm\theta)  \} \, .
$$
Explicitly, given $\bm\theta=(\bm\lambda,\bm\gamma) \in \R^d \times \R^m$, the set  $\Omega_n^{\mathbf f}(\psi,\bm\theta) $ consists of points  $\boldsymbol{\alpha} \in \mathcal U $    such that the system of
inequalities
\begin{equation}\label{sv100}
\left\{
\begin{array}{l}
  \big|\alpha_i -\frac{a_i + \lambda_i}{q}\big|<\frac{\psi(q)}{q}   \qquad 1 \le i \le d  \\[2ex]
  \big|f_j(\boldsymbol{\alpha})-\frac{ b_j + \gamma_j}{q}\big|<\frac{\psi(q)}{q}   \qquad 1 \le j \le m
\end{array}
\right.
\end{equation}
is satisfied for infinitely many   $  (q, \mathbf a, \mathbf b) \in \N  \times \Z^d \times \Z^m $. Furthermore,
there is  no loss of generality in assuming that $ \alq  \in \mathcal U$
for solutions of (\ref{sv100}).   In view of (\ref{bi}),    the sets $\Omega_n^{\mathbf f}(\psi,\bm\theta)$ and    $ \cM_{\mathbf f}\cap \cS_n(\psi,\bm\theta) $ are related by a bi-Lipschitz map and therefore
$$
\hs \left( \cM_{\mathbf f}\cap \cS_n(\psi,\bm\theta) \right) = 0    \quad
\Longleftrightarrow \quad  \hs ( \Omega_n^{\mathbf f}(\psi,\bm\theta)  ) = 0 \ .
$$
Hence, it suffices to show that
\begin{equation}\label{sv101} \hs ( \Omega_n^{\mathbf f}(\psi,\bm\theta)  ) = 0 \ .
\end{equation}

\noindent{\em Step 2. \ } Notice that the set $B=\{ \boldsymbol{\alpha} \in \mathcal U:  \  {\rm l.h.s. \ of \ } (\ref{e:one1})    = 0  \}$ is closed and therefore $G=
\cU\setminus B $ can be written as a countable union of closed rectangles $\cU_i$ on which $f$
satisfies (\ref{e:one1}).   The constant $\eta$ associated with (\ref{e:one1}) depends on the particular choice of $\cU_i$. For the moment, assume that $\hs ( \Omega_n^{\mathbf f}(\psi,\bm\theta) \cap \cU_i  ) = 0$ for any $i \in {\mathbb N}$. On using the fact
that $ \hs(B) = 0 $, we have that
\begin{eqnarray*}\hs ( \Omega_n^{\mathbf f}(\psi,\bm\theta)   )  & \leq &  \hs\Big( B \cup \big(\mbox{ \small
$\bigcup\limits_{i=1}^{\infty}$ } \Omega_n^{\mathbf f}(\psi,\bm\theta) \cap \cU_i \big)\Big)  \\[1ex]
& \leq & \hs( B ) \, + \,  \mbox{ \small
$\sum\limits_{i=1}^{\infty}$ } \hs \big( \Omega_n^{\mathbf f}(\psi,\bm\theta) \cap \cU_i  \big)  \;  = \;  0
\end{eqnarray*}
and this establishes (\ref{sv101}). Thus, without loss of generality,
and  for the sake of clarity we assume that $f$ satisfies
(\ref{e:one1}) on $\mathcal U$.

\noindent{\em Step 3. \ }  For a point $\tfrac{\mathbf p + \bm\theta}{q}\in\R^n$ with  $\mathbf p =  (\mathbf a, \mathbf b) \in  \Z^d \times \Z^m $, let
 $\sigma\big(\tfrac{\mathbf p + \bm\theta}{q}\big)$  denote the set of $\boldsymbol{\alpha} \in \mathcal U $ satisfying (\ref{sv100}).
Trivially, \begin{equation} \label{sv103} {\rm diam}\big(\sigma\big(\tfrac{\mathbf p + \bm\theta}{q}\big)\big)  \,
\ll \psi(q)/q  \, , \end{equation}
where the implied constant depends on $n$ only.

\noindent  Assume that $\sigma\big(\tfrac{\mathbf p + \bm\theta}{q}\big)\not=\emptyset$.  Thus $q$ lies in the integer support $\mathcal N$  of $\psi$.   Let  $  \boldsymbol{\alpha}   \in  \sigma\big(\tfrac{\mathbf p + \bm\theta}{q}\big)  $. The triangle inequality together with  (\ref{bi}) and (\ref{sv100}), implies that
\begin{eqnarray*}
\textstyle{ \big| \mathbf f \big(\alq\big)  -  \frac{\mathbf b + \boldsymbol{\gamma}}{q}\big|  } & \le &
\textstyle{ \big| \mathbf f (\frac{\mathbf a + \boldsymbol{\lambda}}{q})  - \mathbf f (\boldsymbol{\alpha})  \big| \ + \ \big| \mathbf f (\boldsymbol{\alpha})  -  \frac{\mathbf b + \boldsymbol{\gamma}}{q}\big|   }  \\[1ex]
& < &   c_1 \,    \textstyle{ \big| \boldsymbol{\alpha}  - \frac{\mathbf a + \boldsymbol{\lambda}}{q}   \big|    \ +  \ \psi(q)/q   } \\[1ex]
& \le   &
c_2 \, \psi(q)/q   \ , \end{eqnarray*} where  $c_2:= 1+ c_1 $ is a constant. Thus,  for $q$ sufficiently large so that $c_2 \, \psi(q) < 1/2$ we have that
\begin{eqnarray*}
\# \Big\{ \mathbf p      \in \Z^n
 \!\!\!\! &:& \!\!\!\! \sigma(\textstyle{\frac{\mathbf p + \bm\theta}{q} }  )\not=\emptyset \Big\}   \\[2ex] & \leq &
\#  \left\{ \mathbf p      \in \Z^n :  \, \textstyle{\frac{\mathbf a + \bm\lambda}{q} }     \in \mathcal U,  \,
\textstyle{ \big| \mathbf f (\frac{\mathbf a + \boldsymbol{\lambda}}{q})  -  \frac{\mathbf b + \boldsymbol{\gamma}}{q}\big|   <
c_2
\, \psi(q)/q} \right\} \\[2ex]
& = & \# \left\{
\mathbf a  \in \mathbb Z^d  \; : \;  \, \textstyle{\frac{\mathbf a + \bm\lambda}{q} }     \in \mathcal U,  \,
\textstyle{ \|q \, \mathbf f\big(\tfrac{\mathbf a + \bm\lambda}{q}\big) - \bm\gamma \| < c_2 \psi(q) } \, \right\}  \ .
\end{eqnarray*}
By definition, the right hand side is simply the counting function
$ A(q,c_2 \psi,\bm\theta)$. Thus, by Corollary \ref{c:svcor1}, for $q \in \mathcal N$ sufficiently large we have that
\begin{equation}
\label{sv104} \# \Big\{ \mathbf p      \in \Z^n
: \sigma(\textstyle{\frac{\mathbf p + \bm\theta}{q} }  )\not=\emptyset \Big\} \ll \psi(q)^m \, q^{d} \ .
\end{equation}

\noindent{\em Step 4. \ }  For $q >0$, let
$$
\Omega_n^{\mathbf f}(\psi,\bm\theta;q)  \ :=  \bigcup_{\mathbf p      \in \Z^n ,\,\sigma(\frac{\mathbf p + \bm\theta}{q}   )\not=\emptyset}
\!\!\!\!\!\!\!\! \sigma\big(\tfrac{\vv p + \bm\theta}{q}\big)  \  \  .
$$

\noindent Then  $\hs (  \Omega_n^{\mathbf f}(\psi,\bm\theta))  = \hs ( \limsup_{q \to \infty}
\Omega_n^{\mathbf f}(\psi,\bm\theta;q)  )  $  and the  Hausdorff-Cantelli Lemma \cite[p.~68]{BD99} implies
(\ref{sv101}) if
\begin{equation}\label{sv105} \sum_{q=1}^{\infty}  \quad   \sum_{\mathbf p      \in \Z^n ,\,\sigma(\frac{\mathbf p + \bm\theta}{q}   )\not=\emptyset}
\!\!\!\!\!\!\!\! \Big({\rm diam} \big(\sigma\big(\tfrac{\vv p+\bm\theta}{q}\big)\big)\Big)^s < \infty  \, .  \end{equation}
In view of (\ref{sv103}) and (\ref{sv104}), it follows
that
\begin{eqnarray*}
{\rm  L.H.S  \ of \ (\ref{sv105})}  &    \ll &  \sum_{q\in \mathcal N}  \quad   \sum_{\mathbf p      \in \Z^n ,\,\sigma(\frac{\mathbf p + \bm\theta}{q}   )\not=\emptyset}
\!\!\!\!\!\!\!\! (\psi(q)/q)^s
\\ \\ & \ll &
\sum_{q  \in \mathcal N}  (\psi(q)/q)^s \times \psi(q)^m \, q^{d}  \,  \
=  \ \sum_{q=1}^{\infty} (\psi(q)/q)^{s+m} \, q^{n} \  < \ \infty \ \ .
\end{eqnarray*}

\noindent This completes the proof of Theorem \ref{t:genconv}.

%
%

%
%

\vspace{4ex}

\noindent{\bf Acknowledgements.}  We would like to thank the referees for
their  thorough  reports which enabled us to remove a number of
inaccuracies and  improve the clarity of the paper. Thank you both!
SV would like to thank Shoeb Khan for his wonderful friendship and for helping him to keep his curry muscle under control.  Also many thanks to the teenagers,   Ayesha
and Iona, and their mother, Bridget for much happiness, love and laughter.

\end{document}